\documentclass[twoside]{amsart}

\usepackage{amssymb}
\usepackage{amsfonts}
\usepackage{amsmath}
\usepackage{amsthm}
\usepackage{xcolor}
\usepackage{graphicx}
\usepackage{xcolor}

\DeclareMathOperator{\pd}{pd}
\DeclareMathOperator{\PD}{PD}

\newtheorem{theorem}{Theorem}[section]
\newtheorem{lemma}[theorem]{Lemma}

\newtheorem{corollary}[theorem]{Corollary}
\newtheorem{remark}[theorem]{Remark}
\newtheorem{conjecture}[theorem]{Conjecture}

\theoremstyle{definition}
\newtheorem{definition}[theorem]{Definition}

\theoremstyle{remark}

\numberwithin{equation}{section}

\makeatletter 

\title[Partitions with Designated Summands]{Partitions With Designated Summands Not Divisible by $2^l$, $2$, and $3^l$ Modulo $2$, $4$, and $3$}
\author[Herden, Sepanski, Stanfill, $\ldots$]{Daniel Herden, Mark R. Sepanski, Jonathan Stanfill, Cordell Hammon, Joel Henningsen, Henry Ickes, Indalecio Ruiz}
\address{
Department of Mathematics,
Baylor University,
Sid Richardson Building,
1410 S.4th Street,
Waco, TX 76706, USA}
\email{
mark\_sepanski@baylor.edu}
\begin{document}



\begin{abstract}
Numerous congruences for partitions with designated summands have been proven since first being introduced and studied by Andrews, Lewis, and Lovejoy. This paper explicitly characterizes the number of partitions with designated summands whose parts are not divisible by $2^\ell$, $2$, and $3^\ell$ working modulo $2,\ 4,$ and $3$, respectively, greatly extending previous results on the subject. We provide a few applications of our characterizations throughout in the form of congruences and a computationally fast recurrence. Moreover, we illustrate a previously undocumented connection between the number of partitions with designated summands and the number of partitions with odd multiplicities.
\end{abstract}

\maketitle
\tableofcontents


\section{Introduction}\label{intro}

Partitions with designated summands have been studied by many authors since first being introduced by Andrews, Lewis, and Lovejoy in \cite{ALL02}. These partitions are constructed by taking the ordinary partitions and marking exactly one part of each part size, often denoted by affixing a prime to the tagged part. For instance, there are 15 partitions of 5 with designated summands:
\begin{gather*}
5',\ 4'+1',\ 3'+2',\ 3'+1'+1,\ 3'+1+1',\ 2'+2+1',\\
2+2'+1',\ 2'+1'+1+1,\ 2'+1+1'+1,\ 2'+1+1+1', 1'+1+1+1+1,\\
1+1'+1+1+1,\ 1+1+1'+1+1,\ 1+1+1+1'+1,\ 1+1+1+1+1'.
\end{gather*}
Hence $\PD(5)=15$, where $\PD(n)$ denotes the total number of partitions of $n$ with designated summands.

By \cite[Theorem 1]{ALL02}, the generating function for the number of partitions with designated summands whose parts belong to a set of positive integers, $S$, is given by
\begin{align*}
\sum_{n \geq0} \PD_S(n)q^n=\prod_{n\in S}\dfrac{1-q^{6n}}{(1-q^n)(1-q^{2n})(1-q^{3n})}.
\end{align*}
Furthermore, by taking $S_k$ to be the set of positive integers not divisible by $k$, one obtains the generating function for the number of partitions with designated summands whose parts are not divisible by $k$, denoted by $\PD_k(n)$. It is given in \cite[Corollary 3]{ALL02} by
\begin{align} \label{genfunct}
\sum_{n \geq0} \PD_k(n)q^n=\dfrac{(q^6;q^6)_\infty(q^k;q^k)_\infty(q^{2k};q^{2k})_\infty(q^{3k};q^{3k})_\infty}{(q;q)_\infty(q^2;q^2)_\infty(q^3;q^3)_\infty(q^{6k};q^{6k})_\infty},
\end{align}
where $(a;q)_\infty$ represents the \emph{$q$-series} or \emph{$q$-Pochhammer symbol}
\begin{align*}
(a;q)_\infty=\prod_{n\geq0}(1-aq^n).
\end{align*}

The generating function for $\PD_k(n)$, which will be the central object of study for this paper, has been used to prove numerous congruences for partitions with designated summands in \cite{ALL02,CJJS13,dS20b,dS20a,H17,NG16}. In some cases, the generating function for $\PD_k(n)$ is more explicitly known. For instance, one has the following complete characterization of $\PD_2(n)$ modulo $2$.
\begin{theorem}[{\cite[Corollary 10]{ALL02}}] \label{pd2}
For all $n\geq0$,
\begin{align*}
\PD_2(n)\equiv\begin{cases} 1 \, \operatorname{mod} 2, & \text{if } n=0 \text{ or } n=k^2 \text{ for } k\geq 1, \, 3\nmid k\\
0 \, \operatorname{mod} 2, & \text{otherwise}.
\end{cases}
\end{align*}
\end{theorem}

Similarly, a complete characterization of $\PD_3(n)$ modulo $3$ exists.
\begin{theorem}[{\cite[Theorem 2]{dS20b}}] \label{pd3}
For all $n\geq0$,
\begin{align*}
\PD_3(n)\equiv \begin{cases} 1\quad\ \, \operatorname{mod} 3, & \text{if } n=0,\\
r(n) \, \operatorname{mod} 3, & \text{otherwise},
\end{cases}
\end{align*}
where $r(n)$ is the number of solutions to $n = k(k + 1) + 3m(m+ 1) + 1$ in nonnegative integers $k,m$.
\end{theorem}

In this paper, we study further characterizations of $\PD_k(n)$ by means of its generating function. Our main results include characterizing $\PD_{2^\ell}(n)$ mod $2$ in Section~\ref{sec: mod 2}, $\PD_{2}(n)$ mod $4$ in Section \ref{sec: mod 4}, and $\PD_{3^\ell}(n)$ mod $3$ in Section \ref{sec: mod 3}.

In Section~\ref{sec: mod 2}, Theorem \ref{thm: pd 2 to the ell}, we show that
$$ \PD_{2^\ell}(n) \equiv a_{n\ell} \, \operatorname{mod} 2$$ where
$$a_{n \ell}  = \Big|\{\text{solutions to } n = \sum_{m = 0}^{\ell - 1} 2^m k_m^2 \, | \, k_m \geq 0, \, 3 \nmid k_m \text{ or } k_m=0\}\Big|.$$
For example, along with a few auxiliary results, this allows us to validate numerous congruences in Theorem \ref{thm: cong 1A} and Corollary \ref{cor: cong 1D 3.6}.
For $n\geq0$ and \\$r \in\{5,7,10,13,14,15,20,21,23,26,28,29,30,31\}$,
\begin{align*}
	\PD_4(32n+r)\equiv 0 \, \operatorname{mod} 2,
\end{align*}
and
\begin{align*}
	\PD_{8}(32n+24)\equiv 0 \, \operatorname{mod} 2.
\end{align*}
Additionally,
for $\ell \geq 3$ and $r\in\{4, 6\}$,
$$ \PD_{2^\ell}(8n+r) \equiv 0 \, \operatorname{mod}  2,$$
for $\ell \geq 4$ and  $r\in\{4,6,10,12,14\}$
$$ \PD_{2^\ell}(16n+r) \equiv 0 \, \operatorname{mod}  2,$$
and for $\ell \geq 5$ and
$r\in\{4,6,10,12,14,16,20,22,24,26,28,30\}$
$$ \PD_{2^\ell}(32n+r) \equiv 0 \, \operatorname{mod}  2.$$
In general for $\ell \geq j\geq 3$ and $0\leq s < 2^{j-1}$ with $s$ not of the form $0$ or $4^a (8b +1)$ for some $a,b\geq 0$,
$$ \PD_{2^\ell}(2^j n+2s) \equiv 0 \, \operatorname{mod}  2.$$

Finally, in Theorem \ref{thm: odd mult} we provide the very interesting congruence
$$ \PD (n) \equiv b_n \, \operatorname{mod} 2,$$
which links the number $\PD (n)$ of partitions of $n$ with designated summands \cite[A077285]{OEIS}
to the number $b_n$ of partitions of $n$ with odd multiplicities \cite[A055922]{OEIS}, which may be of use in the study of the parity of partitions with odd multiplicities as in \cite{HS18}, \cite{KZ20}, and \cite{SS20}.

In Sections \ref{sec: pd4} and \ref{sec: dissec}, we prove the following remarkably explicit alternate characterization of $\PD_4(n)$ mod $2$ in Theorem \ref{thm: pd4}:
For all $n\geq 0$,
\begin{align*}
	\PD_4(n)\equiv\begin{cases} 1 \, \operatorname{mod} 2, & \text{if } n=m k^2 \text{ for } m,k\geq 0, \, m \mid 6,\\
		0 \, \operatorname{mod} 2, & \text{otherwise}.
	\end{cases}
\end{align*}

In Section \ref{fast recurrence}, motivated by \cite[Corollaries 6 and 9]{ALL02}, \cite{HSJ1etc20}, and Section \ref{sec: mod 2}, for general $n\geq0$ and $k\geq2$, we also provide a computationally fast recurrence for $\PD_k(n)$ mod $2$ in Theorem~\ref{thm: fast rec.} of the form
\begin{align*}
	\PD_k(n) + \sum_{\ell\ge 1,\, 3\nmid \ell} \PD_k(n-\ell^2) \equiv
	\begin{cases} 1 \, \operatorname{mod} 2, & \text{if } n=0 \text{ or } n =km^2 \text{ for some } 3\nmid m,\\
		0 \, \operatorname{mod} 2, & \text{otherwise}.
	\end{cases}
\end{align*}

In Section \ref{sec: mod 4}, we consider $\PD_2(n) \operatorname{mod} 4$. Theorem \ref{Theorem odd}
gives the following explicit characterization of $\PD_2(2n+1)$ mod $4$,
$$ \PD_2(2n+1)\equiv d_n \, \operatorname{mod} 4 $$
where
\begin{align*}
	d_n = &\bigg|\{ \text{solutions to } n = 3j(3j- 1) + 3k(3k- 1) = 2\big[\binom{3j}2 +\binom{3k}2\big]  \, \Big| \, \  k,j\in\mathbb{Z}   \}\bigg|.
\end{align*}
Theorem \ref{pd2(2n) mod4} provides an explicit characterization of $\PD_2(2n)$ mod $4$ as
	\begin{align*}
		\sum_{n\geq0}\PD_2(2n)q^{n} &  \equiv 1+2\sum_{k \geq 1, \, 3 \nmid k}q^{k^2}+\sum_{k,\ell \geq 1, \, 3 \nmid k,\ell}q^{k^2+\ell^2} \, \operatorname{mod} 4.
	\end{align*}
Thoerem \ref{pd2(n) mod4} gives the following combined expression, 
	\begin{align*}
		\pd_2(q) 
		\equiv  1+\Big( \sum_{k \geq 1, \, 3 \nmid k}q^{k^2}\Big) \Big(\sum_{k\in \mathbb{Z}}q^{k^2}\Big) \, \operatorname{mod} 4.
	\end{align*}
Finally, in Theorem \ref{thm: cong 2}, we show
that for $n\geq1,$
\begin{align*}
	\PD_{2}(3n)\equiv0  \, \operatorname{mod}  4
\end{align*}
and for all $n\geq1$ with $6\nmid n$,
\begin{align*}
	\PD_{2}(2n+1)&\equiv0  \, \operatorname{mod}  4.
\end{align*}

In Section \ref{sec: mod 3}, Theorem \ref{our pd3}, we show that
$$ \PD_{3}(n) \equiv e_{1 \ell} \, \operatorname{mod} 3$$
where
\begin{align*}
	e_{n1} &= \big|\{\text{solutions to } n = k_0^2  + 3 k_1^2\, | \, \\
	&\quad\qquad\; k_m \in \mathbb{Z} \text{ or } \mathbb{N} \text{ when $k_m$ is even or odd, respectively}  \}\big|.
\end{align*} This gives an alternate characterization of Theorem \ref{pd3}. See Remark \ref{remark pd3} for alternate forms of $e_{nl}$.

In general, Theorem \ref{thm: cong 4B}, for all $n\geq0$ and $\ell \ge 2$, we show that
\begin{align*}
	\PD_{3^\ell}(3n)\equiv e^*_{n \ell}  \, \operatorname{mod}  3
\end{align*}
where
\begin{align*}
	e^*_{n \ell} & = \big|\{\text{solutions to } n = k_0^2+k_0'^2+ 3^{\ell-1}(k_1^2 +k_1'^2) \, | \, \\
	&\quad\qquad\; k_m,k_m' \in \mathbb{Z} \text{ or } \mathbb{N} \text{ when $k_m,k_m'$ is even or odd, respectively}\}\big|.
\end{align*}

As an easy application, Theorem \ref{thm: cong 3}, we provide very short proofs of the following congruences for which proofs using dissections can be found in \cite[Theorem 3]{dS20b} and \cite[Theorem 3]{dS20a} for $\PD_3(9n+6)$ and $\PD_{3k}(3n+2)$, respectively,
	\begin{align*}
		\PD_{3}(9n+6)\equiv 0 \, \operatorname{mod} 3
	\end{align*}
	and
	\begin{align*}
		\PD_{3^\ell}(3n+2)\equiv 0 \, \operatorname{mod} 3.
	\end{align*}
As a further application, Theorem \ref{thm: cong 4}, we show that for $n\geq1$,
\begin{align*}
	\PD_{3}(2n)\equiv 0 \, \operatorname{mod} 3.
\end{align*}
Moreover, for all $n\geq0$ and $\ell\geq3$,
\begin{align*}
	\PD_{3^\ell}(27n+9)\equiv 0 \, \operatorname{mod} 3
\end{align*}
and for all $n\geq0$ and $\ell\neq2$,
\begin{align*}
	\PD_{3^\ell}(27n+18)\equiv 0 \, \operatorname{mod} 3.
\end{align*}
Finally, Corollary \ref{cor: cong 1D}, for all $n\geq0$, $k\ge 1$ and $\ell\geq 2k+1$, we show that
\begin{align*}
	\PD_{3^{\ell}}\big(3^{2k}(3n+1)\big)\equiv \PD_{3^{\ell}}\big(3^{2k}(3n+2)\big)\equiv 0 \, \operatorname{mod} 3.
\end{align*}

In Section \ref{concluding}, we conclude with some remarks and conjectured congruences.

\section{General Notation and Basic Results}

\begin{definition}
	Let
	$$ f_k = f_k (q) = (q^k; q^k)_\infty
	= \prod_{n \geq 1} (1 - q^{kn})
	= f_1(q^k), $$
	$$	\pd(q) = \sum_{n \geq 0} \PD(n) q^n, $$
	$$ \pd_k(q) = \sum_{n \geq 0} \PD_k(n) q^n, $$
	where, once again, $\PD(n)$ denotes the number of partitions of $n$ with designated summands (ordinary partitions with exactly one part designated among parts with equal size) and $\PD_k(n)$ denotes the number of partitions of $n$ with designated summands whose parts are not divisible by $k$.
\end{definition}

Using this notation, the generating functions can be written as
$$	\pd(q) 	= \frac{f_6(q)}{f_1(q) f_2(q) f_3(q)} $$
and
$$ \pd_k(q) =  \frac{f_6(q)}{f_1(q) f_2(q) f_3(q)}
				\frac{f_k(q) f_{2k}(q) f_{3k}(q)}{f_{6k}(q)}
=  \frac{f_6(q)}{f_1(q) f_2(q) f_3(q)}
				\frac{f_1(q^k) f_{2}(q^k) f_{3}(q^k)}{f_{6}(q^k)}.$$
\begin{definition}
	Let
	$$
	g(q) = \frac{1}{\pd(q)}
	= \frac{f_1(q) f_{2}(q) f_{3}(q)}{f_{6}(q)}. $$
\end{definition}
We see that
\begin{equation} \label{pdk}
	\pd_k(q) = \frac{g(q^k)}{g(q)}.
\end{equation}

Modulo a prime, $p$, the Frobenius automorphism implies that
$$ f_{pk}(q) = f_k(q^p) \equiv f_k(q)^p \, \operatorname{mod} p$$
so that
$$ g(q^p) \equiv g(q)^p \, \operatorname{mod} p.$$
Therefore, Equation \eqref{pdk} gives us
\begin{equation} \label{pdk mod p}
	\pd_p(q) \equiv g(q)^{p-1} \, \operatorname{mod} p.
\end{equation}

Finally, observe that
\begin{equation} \label{pdk p to the ell}
		\pd_{p^\ell}(q) = \frac{g(q^{p^{\ell}})}{g(q)}=\prod_{m = 0}^{\ell - 1} \frac{g(q^{p^{m+1}})}{g(q^{p^{m}})} =
	\prod_{m = 0}^{\ell - 1} \pd_p(q^{p^m}). 	
\end{equation}

\section{The Case of $\pd_{2^\ell}(q) \, \operatorname{mod} 2$}\label{sec: mod 2}

The following can be written more succinctly, but is written to synchronize better with the general case which follows directly after.

\begin{lemma} \label{pd2 mod 2}
	Working modulo $2$,
	$$ \pd_2(q) \equiv g(q) = \frac{1}{\pd(q)}  \, \operatorname{mod} 2$$
	and
	$$ \pd_{2}(q) \equiv \sum_{n \geq0} a_{n1} q^n \, \operatorname{mod} 2$$
	where
	$$ a_{n 1} = \big|\{\text{solutions to } n = k_0^2 \, | \, k_0 \geq 0, \, 3 \nmid k_0 \text{ or } k_0=0\}\big|. $$
\end{lemma}

\begin{proof}
	Using $p = 2$ in Equation \eqref{pdk mod p}, we immediately get
	$$ \pd_2(q) \equiv g(q) = \frac{1}{\pd(q)}  \, \operatorname{mod} 2.$$

	The second statement is a rephrasing of Theorem \ref{pd2},
\begin{equation} \label{pd2 mod 2A}
	\qquad \pd_2(q) \equiv 1 + \sum_{k \geq 1, \, 3 \nmid k} q^{k^2} \, \operatorname{mod} 2. \qedhere
\end{equation}
\end{proof}

The general case is given by the following.

\begin{theorem} \label{thm: pd 2 to the ell}
	For all $n\geq0$ and $\ell \ge 1$,
	$$ \PD_{2^\ell}(n) \equiv a_{n\ell} \, \operatorname{mod} 2$$
	where
	$$ a_{n \ell} = \Big|\{\text{solutions to } n = \sum_{m = 0}^{\ell - 1} 2^m k_m^2 \, | \, k_m \geq 0, \, 3 \nmid k_m \text{ or } k_m=0\}\Big|. $$
\end{theorem}

\begin{proof}
	Using Equations \eqref{pdk p to the ell} and \eqref{pd2 mod 2A}, we see that
	$$ \pd_{2^\ell}(q) =\prod_{m = 0}^{\ell - 1} \pd_2(q^{2^m})
	\equiv \prod_{m = 0}^{\ell - 1} \Big( 1 + \sum_{k \geq 1, \, 3 \nmid k} q^{2^m k^2}\Big) \, \operatorname{mod} 2. $$
	The result follows.
\end{proof}

As a first easy application of Theorem \ref{thm: pd 2 to the ell}, we prove the following newly observed congruences of the form $\PD_{2^\ell}(32n+r)$.

\begin{theorem}\label{thm: cong 1A}
For $n\geq0$ and $r \in\{5,7,10,13,14,15,20,21,23,26,28,29,30,31\}$,
\begin{align*}
\PD_4(32n+r)\equiv 0 \, \operatorname{mod} 2,
\end{align*}
and for $n\geq0$,
\begin{align*}
\PD_{8}(32n+24)\equiv 0 \, \operatorname{mod} 2.
\end{align*}
\end{theorem}

\begin{proof}
The first congruence is immediate as the only quadratic residues modulo $32$ are $0, 1, 4, 9, 16, 17,$ and $25$.

For the second congruence, we need to pair off solutions $(k_0,k_1,k_2)$ of the equation
\begin{align}\label{32n+24,2}
32n+24=k_0^2+2k_1^2+4k_2^2.
\end{align}
Note that we have $k_0^2 \equiv 0$ mod $2$, thus $2\mid k_0$ and we can write $k_0 = 2a$
for some integer $a$. This gives
\[32n+24=4a^2+2k_1^2+4k_2^2.\]
Note that with $(2a,k_1,k_2)$ also $(2k_2,k_1,a)$ is a solution
of \eqref{32n+24,2}, and for $k_2 \ne a$ we will pair these two solutions off. This leaves us with considering
solutions of the form $(2a,k_1,a)$ and the equation
\[32n+24=4a^2+2k_1^2+4a^2 =8a^2+ 2k_1^2.\]
Note that we have $2k_1^2 \equiv 0$ mod $4$, thus $2\mid k_1$ and we can write $k_1 = 2b$
for some integer $b$. This gives the equation
\[4n+3=a^2+b^2\]
which is not satisfiable as the only quadratic residues mod $4$ are $0$ and $1$. Thus, solutions of the form $(2a,k_1,a)$
do not exist and the pairing is complete.
\end{proof}

Our next result expands on and iterates the idea of pairing off solutions.

\begin{theorem}\label{thm: cong 1B}
For $n\geq0$, $\ell \ge 3$, and $0\le s < 2^{\ell-1}$,
\begin{align*}
\PD_{2^\ell}(2^\ell n+2s)\equiv 1  \, \operatorname{mod}  2
\end{align*}
implies that $s$ is a quadratic residue modulo $2^{\ell-1}$. Equivalently, if $0\leq s < 2^{\ell-1}$ and $s$ is quadratic nonresidue modulo $2^{\ell-1}$, then
$$ PD_{2^\ell}(2^\ell n+2s)\equiv 0  \, \operatorname{mod}  2.$$

Moreover, for all $n\geq0$, $\ell \ge 2$, we have
\begin{align*}
\PD_{2^\ell}(2n)\equiv a^*_{n \ell}  \, \operatorname{mod}  2
\end{align*}
where
	$$ a^*_{n \ell} = \big|\{\text{solutions to } n = k_0^2+2^{\ell-1} k_1^2 \, | \, k_m \geq 0, \, 3 \nmid k_m \text{ or } k_m=0\}\big|. $$
\end{theorem}

\begin{proof}
We will be pairing off the solutions $(k_0,k_1,k_2,k_3,\ldots,k_{\ell-1})$ of the equation
\begin{align}\label{2l n+m A}
2^\ell n+2s = k_0^2 + 2 k_1^2 + 4 k_2^2 + 8 k_3^2 + 16 k_4^2 +\ldots + 2^{\ell-1} k_{\ell -1}^2.
\end{align}
Note that we have $k_0^2 \equiv 0$ mod $2$, thus $2\mid k_0$ and we can write $k_0 = 2a$
for some integer $a$. This gives
\[
2^\ell n+2s = 4a^2 + 2 k_1^2 + 4 k_2^2 + 8 k_3^2 + 16 k_4^2 +\ldots + 2^{\ell-1} k_{\ell -1}^2.
\]
Note that with $(2a,k_1,k_2,k_3,\ldots,k_{\ell-1})$ also $(2k_2,k_1,a,k_3,\ldots,k_{\ell-1})$ is a solution
of \eqref{2l n+m A}, and for $k_2 \ne a$ we will pair these two solutions off. This leaves us with considering
solutions of the form $(2a,k_1,a,k_3,\ldots,k_{\ell-1})$ and the equation
\[
2^\ell n+2s = 4a^2 + 2 k_1^2 + 4 a^2 + 8 k_3^2 + 16 k_4^2 +\ldots + 2^{\ell-1} k_{\ell -1}^2.
\]
Note that with $(2a,k_1,a,k_3,k_4\ldots,k_{\ell-1})$ also $(2k_3,k_1,k_3,a,k_4\ldots,k_{\ell-1})$ is a solution
of \eqref{2l n+m A}, and for $k_3 \ne a$ we will pair these two solutions off. This leaves us with considering
solutions of the form $(2a,k_1,a,a,k_4\ldots,k_{\ell-1})$ and the equation
\[ 2^\ell n+2s = 4a^2 + 2 k_1^2 + 4 a^2 + 8 a^2 + 16 k_4^2 +\ldots + 2^{\ell-1} k_{\ell -1}^2.\]
This process can be iterated until we are left to consider solutions of the form $(2a,k_1,a,a,\ldots,a)$
and the equation
\[ 2^\ell n+2s = 4a^2 + 2 k_1^2 + 4 a^2 + 8 a^2 + 16 a^2 +\ldots + 2^{\ell-1} a^2 = 2k_1^2 + 2^\ell a^2.\]
In particular,
\[ 2^{\ell-1} n+s = k_1^2 + 2^{\ell-1} a^2.\]
Thus, $s \equiv k_1^2$ mod $2^{\ell-1}$ and the first part of the theorem follows. The second part of the theorem is now easy.
\end{proof}

We continue with a general auxiliary result.

\begin{lemma}\label{lem: cong 1C}
If there exist $j$ and $r$ such that the congruence
\begin{align*}
\PD_{2^j}(2^j n+r)\equiv0  \, \operatorname{mod}  2
\end{align*}
holds for all $n\ge 0$, then
\begin{align*}
\PD_{2^{\ell}}(2^j n+r)\equiv0  \, \operatorname{mod}  2
\end{align*}
for all $n\ge 0$ and $\ell \ge j$.
\end{lemma}

\begin{proof}
Let $j,\ell,n$ and $r$ be fixed. With
\[ L = \Big\{\text{solutions to } 2^j n+r = s +\sum_{m = j}^{\ell - 1} 2^m k_m^2 \, | \, s, k_m \geq 0, \, 3 \nmid k_m \text{ or } k_m=0\Big\} \]
we have
\[ a_{2^j n+r, \ell} = \sum \big\{a_{sj} \, | \, (s,k_j,k_{j+1},\ldots,k_{\ell-1})\in L \big\}, \]
where $s \equiv r$ mod $2^j$ for all $(s,k_j,k_{j+1},\ldots,k_{\ell-1})\in L$.  Hence,
\[ a_{sj} \equiv \PD_{2^j}(s)\equiv0  \, \operatorname{mod} 2\]
for all $(s,k_j,k_{j+1},\ldots,k_{\ell-1})\in L$, thus
\[\PD_{2^{\ell}}(2^j n+r) \equiv a_{2^j n+r, \ell}\equiv0  \, \operatorname{mod} 2.\qedhere \]
\end{proof}

As a demonstration of how Theorem \ref{thm: cong 1B} and Lemma \ref{lem: cong 1C} combine to efficiently prove congruences,
we provide the following corollary. Some of these results were proven with dissections in \cite[Equations (3), (5), and (8)]{dS20a}.

\begin{corollary}\label{cor: cong 1D 3.6}
	Let $n\geq 0$.
	For $\ell \geq 3$ and $r\in\{4, 6\}$,
	$$ \PD_{2^\ell}(8n+r) \equiv 0 \, \operatorname{mod}  2.$$
	For $\ell \geq 4$ and  $r\in\{4,6,10,12,14\}$
	$$ \PD_{2^\ell}(16n+r) \equiv 0 \, \operatorname{mod}  2.$$
	For $\ell \geq 5$ and
	$r\in\{4,6,10,12,14,16,20,22,24,26,28,30\}$
	$$ \PD_{2^\ell}(32n+r) \equiv 0 \, \operatorname{mod}  2.$$
	In general for $\ell \geq j\geq 3$ and $0\leq s < 2^{j-1}$ with $s$ not of the form $0$ or $4^a (8b +1)$ for some $a,b\geq 0$,
	$$ \PD_{2^\ell}(2^j n+2s) \equiv 0 \, \operatorname{mod}  2.$$
\end{corollary}

\begin{proof}
	We apply Theorem \ref{thm: cong 1B} first with $\ell=3$. The quadratic nonresidues modulo~$4$ are $2$ and $3$. Thus for $n\geq 0$ and $r\in\{4, 6\}$,
	$$ \PD_8(8n+r) \equiv 0 \, \operatorname{mod}  2.$$
	
	For $\ell=4$, the quadratic nonresidues modulo $8$ are $2,3,5,6,$ and $7$. So for $r\in\{4,6,10,12,14\}$
	$$ \PD_{16}(16n+r) \equiv 0 \, \operatorname{mod}  2.$$
	
	For $\ell=5$, the quadratic nonresidues modulo $16$ are $2,3,5,6,7,8,10,11,12,13,14,$ and $15$. So for
	$r\in\{4,6,10,12,14,16,20,22,24,26,28,30\}$
	$$ \PD_{32}(32n+r) \equiv 0 \, \operatorname{mod}  2.$$
	
	In general, for $j \geq 3$, the quadratic residues mod $2^{j-1}$ are $0$ and all numbers of the form $4^a(8b+1)$ for $a,b\geq 0$. So if $0\leq s < 2^{j-1}$ is not of this form, then
	$$ \PD_{2^j}(2^j n+2s) \equiv 0 \, \operatorname{mod}  2.$$
	
	The rest follows from Lemma \ref{lem: cong 1C}.
\end{proof}

The function $ \pd_2(q) \equiv \frac{1}{\pd(q)}  \, \operatorname{mod} 2$ is well known. We end this section on an interesting new interpretation of its reciprocal.

\begin{theorem}\label{thm: odd mult}
	Working modulo $2$,
	$$ \pd(q) \equiv \frac{1}{\pd_2(q)} \equiv 1 + \sum_{n \geq 1} b_n q^n \, \operatorname{mod} 2$$
	where
	$$ b_n =\big|\{ \text{partitions of } n \, | \, \text{the multiplicity of each part is odd}\} \big|.$$
\end{theorem}

\begin{proof}
	Observe that
	\begin{eqnarray*}
\frac{1-q^6}{(1-q)(1-q^2)(1-q^3)} & \equiv & \frac{(1-q^3)^2}{(1-q)(1-q^2)(1-q^3)}\\
& = & \frac{1+q+q^2}{1-q^2} \equiv 1 + \sum_{k \geq 0} q^{2k+1} \, \operatorname{mod} 2.
	\end{eqnarray*}
	Using the above equation in conjunction with the definition of $\pd(q)$ in terms of the $f_k$, we get
	$$ \pd(q) \equiv \prod_{n \geq 1} \Big(1 + \sum_{k \geq 0} q^{(2k+1)n}\Big) \, \operatorname{mod} 2$$
	and the result follows.	
\end{proof}

Theorem \ref{thm: odd mult} provides the hitherto undocumented congruence
$$ \PD (n) \equiv b_n \, \operatorname{mod} 2,$$
which links the number $\PD (n)$ of partitions of $n$ with designated summands \cite[A077285]{OEIS}
to the number $b_n$ of partitions of $n$ with odd multiplicities \cite[A055922]{OEIS}.
This may be of use in understanding the parity of partitions with odd multiplicities as studied in \cite{HS18}, \cite{KZ20}, and \cite{SS20}.

\section{An Alternate Characterization of $\pd_4(q) \, \operatorname{mod} 2$} \label{sec: pd4}

In the case of $\ell = 2$, Theorem \ref{thm: pd 2 to the ell} says
$$ \PD_{4}(n) \equiv a_{n2} \, \operatorname{mod} 2$$
where
$$ a_{n 2} = \big|\{\text{solutions to } n = k_0^2 + 2 k_1^2  \, | \, k_m \geq 0, \, 3 \nmid k_m \text{ or } k_m =0 \}\big|. $$
In this section we give a remarkably explicit formula for $a_{n 2} \, \operatorname{mod} 2$ together with a combinatorial proof.
A dissection proof for the same result is provided in Section~\ref{sec: dissec}.

\begin{theorem}\label{thm: pd4}
	For all $n\geq 0$,
	\begin{align*}
		a_{n 2}\equiv\begin{cases} 1 \, \operatorname{mod} 2, & \text{if } n=m k^2 \text{ for } m,k\in\mathbb{Z}_{\geq 0} \text{ with } m \mid 6,\\
			0 \, \operatorname{mod} 2, & \text{otherwise}.
		\end{cases}
	\end{align*}
\end{theorem}

\begin{proof}

Let
$$L(n) = \big\{(k_0,k_1) \, \mid \,  n = k_0^2 + 2 k_1^2, \, k_m \geq 0, \, 3 \nmid k_m \text{ or } k_m =0 \big\} $$
so that $a_{n 2} = | L(n) |$.
Recall that $0,1$ are the only quadratic residues $\, \operatorname{mod} 3$.

\noindent\textbf{Case 1:} $n\equiv 1\, \operatorname{mod} 3$

For $(a,b)\in L(n)$, $n=a^2+2b^2\equiv 1\, \operatorname{mod} 3$ is only possible for $b=0$. Thus,
\begin{align*}
	|L(n)|=\begin{cases} 1, & \text{if } n=k^2\text{ for some } k\in\mathbb{Z}_{\geq0},\\
		0, & \text{otherwise},
	\end{cases}
\end{align*}
confirming the theorem for $n\equiv1\, \operatorname{mod} 3$.

\noindent\textbf{Case 2:} $n\equiv 2\, \operatorname{mod} 3$

For $(a,b)\in L(n)$, $n=a^2+2b^2\equiv 2\, \operatorname{mod} 3$ is only possible for $a=0$. Thus,
\begin{align*}
	| L(n) | = 	\begin{cases} 1, & \text{if } n=2k^2\text{ for some } k\in\mathbb{Z}_{\geq0},\\
		0, & \text{otherwise},
	\end{cases}
\end{align*}
confirming the theorem for $n\equiv2\, \operatorname{mod} 3$.

\noindent\textbf{Case 3:} $n\equiv 0\, \operatorname{mod} 3$

The case $n\equiv0\, \operatorname{mod} 3$ is more involved.
For $n=0$, we have $L(0) = \{ (0,0)\}$ and $a_{02} = |L(0)| =1$, confirming the theorem. Thus we may assume $n>0$.
Let
\begin{align*}
	S(n)=\{(k_0,k_1) \, \mid \, n= k_0^2+2k_1^2,\, k_m \in\mathbb{Z} \}.
\end{align*}
We define an equivalence relation on $S(n)$ by
\begin{align*}
	(a,b)\sim (c,d) \text{ iff } |a|=|c| \text{ and } |b|=|d|,\text{ for } (a,b),(c,d)\in S(n).
\end{align*}
Next, we investigate the equivalence classes of $S(n)/\sim$.

For $(a,b)\in S(n),\ n=a^2+2b^2\equiv 0\, \operatorname{mod} 3$ is only possible for $a,b\equiv 0\, \operatorname{mod} 3$ or $a,b\not\equiv 0 \, \operatorname{mod} 3$. In particular, in this case,
\begin{align*}
	L(n)= \{(k_0,k_1) \, \mid \,  n = k_0^2 + 2 k_1^2,   \, k_m > 0, \, 3 \nmid k_m \},
\end{align*}
and $L(n)$ contains exactly one representative for each equivalence class $[(a,b)]\in S(n)/\sim$ with $a,b\not\equiv0\, \operatorname{mod} 3$.

Noting
\begin{align*}
	a^2+2b^2=n\quad \iff\quad \left(\dfrac{a+4b}{3}\right)^2+2\left(\dfrac{2a-b}{3}\right)^2=n,
\end{align*}
we define $\Psi:\mathbb{R}^2\to\mathbb{R}^2$ by
\begin{align*}
	\Psi(a,b)=\left(\dfrac{a+4b}{3},\dfrac{2a-b}{3}\right).
\end{align*}
Observe that $\Psi$ is an involution, i.e., $\Psi^2=\text{id}$.

We define a graph whose vertices are the equivalence classes of $S(n)/\sim$ by connecting $[(a,b)]$ with $[(c,d)]$ if $\Psi(a',b')\in[(c,d)]$ for some $(a',b')\in[(a,b)]$. Since $\Psi$ is an involution, $[(a,b)]$ is connected with $[(c,d)]$ if and only if $[(c,d)]$ is connected with $[(a,b)]$, and the resulting graph is undirected. We will further investigate this graph. We begin by discussing the most common occurrences, treating exceptions last.

The equivalence classes $[(a,b)]\in S(n)/\sim$ with $a,b\not\equiv0\, \operatorname{mod}3$ are in one-to-one correspondence with $L(n)$. With exceptions to follow below, each such equivalence class is usually connected with exactly one other equivalence class. In particular, for $a\not\equiv b\, \operatorname{mod}3$, $[(a,b)]$ is connected to $[\Psi(a,b)]$ and for $a\equiv b\, \operatorname{mod}3$, $[(a,b)]$ is connected to $[\Psi(a,-b)]$.
Similarly, with exceptions to follow, the equivalence classes $[(a,b)]\in S(n)/\sim$ with $a,b\equiv0\, \operatorname{mod}3$ are usually connected with exactly two other equivalence classes, $[\Psi(a,b)]$ and $[\Psi(a,-b)]$.

Thus, the resulting graph will usually separate the equivalence classes corresponding to $L(n)$ into connected pairs, hence $|L(n)|\equiv 0\, \operatorname{mod}2$.

Exceptions to the above happen if either of the following two cases occurs:
\begin{itemize}
\item[$(1)$] One equivalence class is connected to itself, that is, $[\Psi(a,b)]=[(a,b)]$. This case occurs if and only if
\begin{align*}
	\left|\frac{2a-b}{3}\right|=|b| \iff a=-b\text{ or }a=2b \iff (k,k) \text{ or } (2k,k)\in S(n),
\end{align*}
that is, if and only $n=3k^2$ or $6k^2$.

\item[$(2)$] One equivalence class connects to the same equivalence class twice, that is, $[\Psi(a,b)]=[\Psi(a,-b)]$. This case occurs if and only if
\begin{align*}
	\left|\frac{2a-b}{3}\right|=\left|\frac{2a+b}{3}\right| \iff a=0\text{ or }b=0 \iff (0,k) \text{ or } (k,0)\in S(n),
\end{align*}
that is, if and only $n=2k^2$ or $k^2$ with $k\equiv 0\, \operatorname{mod} 3$.
\end{itemize}
This completes the proof of the Theorem.
\end{proof}

\begin{remark} \label{rem:4.2}\mbox{}
The solution set of the Diophantine equation $a^2+2b^2=n$ can be further investigated using unique factorization in the principal ideal domain $\mathbb{Z}\big[\sqrt{-2}\big]$.
\end{remark}

\section{A Recurrence Relation for $\PD_k(n) \, \operatorname{mod} 2$}\label{fast recurrence}

While Theorem \ref{thm: pd 2 to the ell} provides an explicit description of $\PD_{2^\ell}(n)$ mod $2$, for general $\PD_{k}(n)$ mod $2$ we have the following computationally fast recurrence.

\begin{theorem}\label{thm: fast rec.}
	For $n\geq 0$ and $k\ge 2$,
	\begin{align*}
        \PD_k(n) + \sum_{\ell\ge 1,\, 3\nmid \ell} \PD_k(n-\ell^2) \equiv
        \begin{cases} 1 \, \operatorname{mod} 2, & \text{if }
        	n=0 \text{ or } n=k m^2, \,\, m\geq 1, 3 \nmid m,\\
0 \, \operatorname{mod} 2, & \text{otherwise}.
\end{cases}
	\end{align*}
\end{theorem}

\begin{proof}
Combining Equation \eqref{pdk} with Lemma \ref{pd2 mod 2} gives
\begin{align*}
\pd_k(q) = \frac{g(q^k)}{g(q)} \equiv \frac{\pd_2(q^k)}{\pd_2(q)} \, \operatorname{mod} 2.
\end{align*}
Substituting Equation \eqref{pd2 mod 2A}, we have
\begin{equation}
	\left( \sum_{n \geq 0} \PD_k(n) q^n \right) \left( 1 + \sum_{\ell \geq 1, \, 3 \nmid \ell} q^{\ell^2}\right) \equiv  1 + \sum_{m \geq 1, \, 3 \nmid m} q^{km^2}\, \operatorname{mod} 2,
\end{equation}
and the result follows.
\end{proof}

\section{The Case of $\pd_{2}(q) \, \operatorname{mod} 4$}\label{sec: mod 4}

Next we turn to $\pd_{2}(q) \, \operatorname{mod} 4$ beginning with a helpful result.

\begin{lemma}
Working modulo $4$,
	\begin{equation} \label{f3 3 mod 4}
		f_1^3 \equiv \sum_{n \geq0} q^{\frac{n(n+1)}{2}} \, \operatorname{mod} 4
	\end{equation}
and
	\begin{equation} \label{f8 3 mod 4}
		qf_8^3 \equiv \sum_{n \geq 1,\, 2\nmid n} q^{n^2} \, \operatorname{mod} 4.
	\end{equation}
\end{lemma}

\begin{proof}
	For any positive integers $k$ and $m$, we have the identity,  \cite[Lemma 5]{dS20b},
	\begin{align}\label{mod 4 relation}
		f_{2k}^{2m} \equiv f_k^{4m} \, \operatorname{mod} 4.
	\end{align}
Therefore, the theta identity \cite[Equation (2.2.13)]{partitions} provides
\begin{align*}
\sum_{n\ge 0} q^{\frac{n(n+1)}{2}} =\dfrac{(q^2;q^2)_\infty}{(q;q^2)_\infty}= \frac{f_2^2}{f_1} \equiv f_1^3 \, \operatorname{mod} 4,
\end{align*}
and the second identity is an immediate consequence of
\[
qf_8^3 = q[f_1(q^8)]^3\equiv q \sum_{n\ge 0} q^{4n(n+1)} =\sum_{n\ge 0} q^{(2n+1)^2} \, \operatorname{mod} 4. \qedhere
\]
\end{proof}

This result allows one to prove the following characterization:

\begin{theorem}
For all $n\geq0$,
	$$
	\PD_2(n) \equiv c_n \, \operatorname{mod} 4$$
	where
	\begin{align*}
		c_n & = \Big|\{ \text{solutions to } n = \frac{3k_0(k_0+1)}{2} + \sum_{m\geq 1} k_m(12a_m+b_m)  \, | \, \\
		&\quad\qquad\; k_m,a_m \geq 0, \, b_m\in\{0,1,2,3\}, \; 1\leq k_1 < k_2 < k_3 < \ldots    \}\Big|.
	\end{align*}
\end{theorem}

\begin{proof}
With Equation \eqref{mod 4 relation}, we have the generating function
	\begin{align*}
		\pd_2(q)=\dfrac{f_{4} f_{6}^2}{f_1 f_3 f_{12}}\equiv \dfrac{f_3^3 f_{4}}{f_1 f_{12}}\, \operatorname{mod}4.
	\end{align*}
The result follows now easily from
\[f_3^3 =[f_1(q^3)]^3\equiv \sum_{k \geq0} q^{\frac{3k(k+1)}{2}} \, \operatorname{mod} 4\]
and
\[ 	\frac{1-x^4}{(1-x)(1-x^{12})} = \frac{1}{1+x^4+x^8} \sum_{n\ge 0} x^n =
		\sum_{\substack{n=12a+b,\\ a \geq 0,\\ b\in\{0,1,2,3\}}} x^n.\qedhere \]
\end{proof}

We further prove an explicit characterization of $\PD_2(2n+1)$ modulo $4$ related to generalized pentagonal numbers by means of the following dissection of $\pd_2(q)$ into even and odd powers.

\begin{lemma}\label{lemma-odd/even}
	Working modulo $4$,
\begin{align}\label{odd/even}
	\pd_2(q)\equiv \left[\dfrac{f_2^3}{f_6}\right]^2+qf_{12}^2 \, \operatorname{mod}4,
\end{align}
splitting $\pd_2(q)$ into even and odd powers of $q$,   respectively.
\end{lemma}
\begin{proof}
With Equation \eqref{mod 4 relation}, we have
\begin{align*}
	\pd_2(q)=\dfrac{f_{4} f_{6}^2}{f_1 f_3 f_{12}}\equiv \dfrac{f_3^3 f_{4}}{f_1 f_{12}}\, \operatorname{mod}4.
\end{align*}
Using the identity (cf., \cite[Equation (3.75)]{XY13})
\begin{align*}
	\dfrac{f_3^3}{f_1}=\dfrac{f_4^3 f_6^2}{f_2^2 f_{12}}+q\dfrac{f_{12}^3}{f_4},
\end{align*}
and repeatedly using \eqref{mod 4 relation} yields
\begin{align*}
	\pd_2(q)\equiv \dfrac{f_3^3 f_{4}}{f_1 f_{12}} \equiv \dfrac{f_4^4 f_6^2}{f_2^2 f^2_{12}}+q f_{12}^2 \equiv \left[\dfrac{f_2^3}{f_6}\right]^2+qf_{12}^2 \, \operatorname{mod}4.
\end{align*}
Noting that $f_2^3/f_6$ and $f_{12}$ are power series in $q^2$, the first term on the right-hand side gives all even power contributions while the second term gives all odd power contributions due to the multiplication by the extra factor of $q$.
\end{proof}

An application of Lemma \ref{lemma-odd/even} gives the following explicit characterization of $\PD_2(2n+1)$ mod $4$.

\begin{theorem}\label{Theorem odd}
For all $n\geq0$,
\begin{align*}
\PD_2(2n+1)\equiv d_n \, \operatorname{mod} 4,
\end{align*}
	where
	\begin{align*}
		d_n = &\bigg|\{ \text{solutions to } n = 3j(3j- 1) + 3k(3k- 1) = 2\big[\binom{3j}2 +\binom{3k}2\big]  \, \Big| \, \  k,j\in\mathbb{Z}   \}\bigg|.
	\end{align*}
\end{theorem}
\begin{proof}
One has from \eqref{odd/even},
\begin{align*}
\sum_{n \geq0} \PD_2(2n+1)q^n\equiv f_6^2= [f_1(q^6)]^2=\sum_{j,k\in\mathbb{Z}} (-1)^{j+k}q^{3[j(3j-1)+k(3k-1)]}  \, \operatorname{mod} 4,
\end{align*}
where we have used Euler's pentagonal number theorem \cite[Corollary 1.7]{partitions}. The result follows from noting that on the right-hand side the only contributions with coefficient $-1$ occur with different parities of $j$ and $k$, hence will come in pairs of solutions $(j,k)$, $(k,j)$ each time and thus contribute $2(-1)^{j+k}\equiv 2 \operatorname{mod} 4$ to the count.
\end{proof}

Similarly, Lemma \ref{lemma-odd/even} also provides an explicit characterization of $\PD_2(2n)$ mod~$4$.

\begin{theorem}\label{pd2(2n) mod4}
\begin{align*}
\sum_{n\geq0}\PD_2(2n)q^{n} & \equiv  \left[\dfrac{f_1^3}{f_3}\right]^2 \equiv \Big(1+\sum_{k \geq 1, \, 3 \nmid k}q^{k^2}\Big)^2\\
&  \equiv 1+2\sum_{k \geq 1, \, 3 \nmid k}q^{k^2}+\sum_{k,\ell \geq 1, \, 3 \nmid k,\ell}q^{k^2+\ell^2} \, \operatorname{mod} 4.
\end{align*}
\end{theorem}

\begin{proof}
Note that by \cite[Equations (2) and (4)]{SS20}, we have
\begin{align*}
\dfrac{f_1^3}{f_3} \equiv 1+q\dfrac{f_9^3}{f_3}\equiv 1+ q\sum_{j\in\mathbb{Z}}q^{3j(3j-2)} = 1+ \sum_{j\in\mathbb{Z}}q^{(3j-1)^2}= 1+ \sum_{k \geq 1, \, 3 \nmid k}q^{k^2}\, \operatorname{mod} 2,
\end{align*}
which implies
\begin{align} \label{conjecture pd2(2n) mod4A}
\left[\dfrac{f_1^3}{f_3}\right]^2 \equiv \Big(1+ \sum_{k \geq 1, \, 3 \nmid k}q^{k^2}\Big)^2\, \operatorname{mod} 4.
\end{align}
Combining \eqref{odd/even} with Equation \eqref{conjecture pd2(2n) mod4A} gives
\begin{align*}
\sum_{n\geq0}\PD_2(2n)q^{n} & \equiv  \left[\dfrac{f_1^3}{f_3}\right]^2 \equiv
 \Big(1+\sum_{k \geq 1, \, 3 \nmid k}q^{k^2}\Big)^2 \, \operatorname{mod} 4. \qedhere
\end{align*}
\end{proof}

As an easy application of these results, we prove a few observed congruences. Alternative proofs using dissections can be found in \cite[Corollary 1.4, Theorem 1.5]{BO15}.

\begin{theorem}\label{thm: cong 2}
For $n\geq1,$ we have
\begin{align*}
\PD_{2}(3n)\equiv0  \, \operatorname{mod}  4.
\end{align*}
 For all $n\geq1$ with $6\nmid n$, we have
\begin{align*}
\PD_{2}(2n+1)&\equiv0  \, \operatorname{mod}  4.
\end{align*}
\end{theorem}
\begin{proof}
The first congruence is easily seen by noting from \cite[Theorem 22]{ALL02} and Equation \eqref{mod 4 relation} that
\begin{align*}
\sum_{n \geq0} \PD_2(3n)q^n&=\dfrac{f_2^2 f_{6}^4}{f_1^4  f_{12}^2}\equiv\dfrac{f_1^4 f_{6}^4}{f_1^4  f_{6}^4}\equiv 1 \operatorname{mod} 4.
\end{align*}
The second congruence is an immediate consequence of Theorem \ref{Theorem odd} as
\[n = 3j(3j- 1) + 3k(3k- 1) = 2\big[\binom{3j}2 +\binom{3k}2\big]\]
implies that $n \equiv 0$ mod $6$.
\end{proof}

We end this section with yet another nice characterization of $\PD_2(n)$ mod $4$.

\begin{theorem}\label{pd2(n) mod4}
\begin{align*}
\pd_2(q) \equiv 1+ \Big( \sum_{k \geq 1, \, 3 \nmid k}q^{k^2}\Big) \Big(1+2 \sum_{k\ge 1}q^{k^2}\Big) = 1+\Big( \sum_{k \geq 1, \, 3 \nmid k}q^{k^2}\Big) \Big(\sum_{k\in \mathbb{Z}}q^{k^2}\Big) \, \operatorname{mod} 4.
\end{align*}
\end{theorem}

\begin{proof}
Starting from \cite[Equation (3.13)]{ALL02}, we have
\begin{align*}
\pd_2(q) & =  \frac{\sum_{j \in \mathbb{Z}}q^{(3j)^2}- \sum_{j \in \mathbb{Z}}q^{(3j+1)^2}}{1+2 \sum_{k\ge 1}(-1)^k q^{k^2}}
=  \frac{1+ 2\sum_{k \geq 1, \, 3 \mid k}q^{k^2} - \sum_{k \geq 1, \, 3 \nmid k}q^{k^2}}{1+2 \sum_{k\ge 1}(-1)^k q^{k^2}}\\
&  \equiv \frac{1+ 2\sum_{k \geq 1, \, 3 \mid k}q^{k^2} + 3\sum_{k \geq 1, \, 3 \nmid k}q^{k^2}}{1+2 \sum_{k\ge 1}q^{k^2}}
=\frac{1+2 \sum_{k\ge 1}q^{k^2} + \sum_{k \geq 1, \, 3 \nmid k}q^{k^2}}{1+2 \sum_{k\ge 1}q^{k^2}}\\
& = 1+ \frac{\sum_{k \geq 1, \, 3 \nmid k}q^{k^2}}{1+2 \sum_{k\ge 1}q^{k^2}}
\equiv 1+ \Big( \sum_{k \geq 1, \, 3 \nmid k}q^{k^2}\Big) \Big(1+2 \sum_{k\ge 1}q^{k^2}\Big)\, \operatorname{mod} 4.\qedhere
\end{align*}
\end{proof}

\section{A Dissection Proof for Theorem \ref{thm: pd4}}\label{sec: dissec}

In this section, we will prove Theorem \ref{thm: pd4} with the help of dissections. This proof will reuse some of the identities from the last section. In addition, we will also need the following auxiliary result.

\begin{lemma}
	Working modulo $2$,
	\begin{align}\label{f1f3}
		q^2f^6_2f^6_6 \equiv q^2f^3_{16} + q^6f^3_{48}\, \operatorname{mod}2.
	\end{align}
\end{lemma}

\begin{proof}
	Using the identity (cf., \cite[Equation (3.12)]{XY13})
	\begin{align*}
		\dfrac{1}{f_1f_3}=\dfrac{f_8^2f_{12}^5}{f_2^2 f_4f_6^4f_{24}^2}+q\dfrac{f_4^5f_{24}^2}{f_2^4f_6^2f_8^2f_{12}},
	\end{align*}
	we have
	\begin{align*}
		f^3_1f^3_3 = \dfrac{f^4_1f^4_3 f_8^2f_{12}^5}{f_2^2 f_4f_6^4f_{24}^2}+q\dfrac{f^4_1f^4_3 f_4^5f_{24}^2}{f_2^4f_6^2f_8^2f_{12}}\equiv f_4^3 +qf_{12}^3\, \operatorname{mod}2,
	\end{align*}
	where we make repeated use of the identity $f_{2k} \equiv f_k^2 \, \operatorname{mod} 2$. Raising the last equation to the fourth power gives
	\begin{align*}
		f^6_2f^6_6 \equiv f_{16}^3 +q^4f_{48}^3\, \operatorname{mod}2,
	\end{align*}
	and the result follows.
\end{proof}

\noindent\textit{Alternate proof of Theorem \ref{thm: pd4}.} For a proof of Theorem \ref{thm: pd4}, we start with the generating function
\begin{align*}
	\pd_4(q)=\dfrac{f_{4}f_{6}f_8f_{12}}{f_1 f_2 f_3 f_{24}}\equiv \left[\dfrac{f_1^3}{f_3}\right]^3\, \operatorname{mod}2.
\end{align*}

Applying Equation \eqref{pdk p to the ell} and Lemma \ref{lemma-odd/even} one can provide the even and odd power dissection of $\pd_4(q)$,
\begin{align}
	\pd_4(q)&= \pd_2(q)\pd_2(q^2) \notag \\
	&\equiv \left(\left[\dfrac{f_2^3}{f_6}\right]^2+qf_{12}^2\right)\!\! \left(\left[\dfrac{f_4^3}{f_{12}}\right]^2+q^2f_{24}^2\right)
	\equiv \left(\left[\dfrac{f_2^3}{f_6}\right]^2+qf_{6}^4\right)\!\!\left(\left[\dfrac{f_2^3}{f_{6}}\right]^4+q^2f_{6}^8\right) \notag \\
	&\equiv \left[\dfrac{f_2^3}{f_6}\right]^6+q^2f_2^6 f_6^6+q f_2^{12}+q^3f_6^{12} \notag \\
	&\equiv \left[\dfrac{f_2^3}{f_6}\right]^6+q^2f_{16}^3+q^6f_{48}^3+q f_{8}^{3}+q^3f_{24}^{3} \,  \operatorname{mod} 2, \label{pd_4 dissection}
\end{align}
where the last line is a result of Equation \eqref{f1f3} and once again applying the identity $f_{2k} \equiv f_k^2 \, \operatorname{mod} 2$.
In Equation \eqref{pd_4 dissection}, noting that
\[
\left[\dfrac{f_2^3}{f_6}\right]^6\equiv \left[\dfrac{f_1^3}{f_3}\right]^{12}\equiv [\pd_4(q)]^4 \equiv \pd_4(q^4) \, \operatorname{mod} 2
\]
and applying Equation \eqref{f8 3 mod 4} to $q^r f_{8r}^3=q^r [f_8(q^r)]^3$, $r\in\{1,2,3,6\},$ yields
\begin{align*}
	\pd_4(q)\equiv \pd_4(q^4)+\sum_{n\geq1,\, 2\nmid n} \big[ q^{n^2}+ q^{2n^2}+ q^{3n^2}+ q^{6n^2}\big] \, \operatorname{mod} 2.
\end{align*}
Raising this equation to the fourth power yields
\begin{align*}
	\pd_4(q^4)\equiv \pd_4(q^{16})+\sum_{n\geq1,\, 2\nmid n} \big[ q^{(2n)^2}+ q^{2(2n)^2}+ q^{3(2n)^2}+ q^{6(2n)^2} \big] \, \operatorname{mod} 2,
\end{align*}
and combining the last two equations results in
\begin{align*}
	\pd_4(q)\equiv \pd_4(q^{4^2})+\sum_{\substack{n\geq1,\, 2\nmid n,\\ i\, \in \{0,1\}}} \big[ q^{(2^i n)^2}+ q^{2(2^i n)^2}+ q^{3(2^i n)^2}+ q^{6(2^i n)^2} \big] \, \operatorname{mod} 2.
\end{align*}

Therefore, iterating one obtains, for $\ell \ge 1$,
\begin{align*}
	\pd_4(q)\equiv \pd_4(q^{4^\ell})+\sum_{\substack{n\geq1,\, 2\nmid n,\\ 0\le i<\ell}} \big[ q^{(2^i n)^2}+ q^{2(2^i n)^2}+ q^{3(2^i n)^2}+ q^{6(2^i n)^2} \big] \, \operatorname{mod} 2,
\end{align*}
where the first summand on the right-hand side contains only powers of $q^{4^\ell}$. For $\ell \to \infty$, this results in the identity
\begin{align*}
	\pd_4(q) &\equiv 1 +\sum_{n\geq1,\, 2\nmid n,\, i\ge 0} \big[ q^{(2^i n)^2}+ q^{2(2^i n)^2}+ q^{3(2^i n)^2}+ q^{6(2^i n)^2} \big]\\
	&\equiv 1+\sum_{n\geq1} \big[ q^{n^2}+ q^{2n^2}+ q^{3n^2}+ q^{6n^2}\big] \, \operatorname{mod} 2,
\end{align*}
completing the proof of Theorem \ref{thm: pd4}.
\qed\medskip\\
We will end this section with a few Rogers-Ramanujan-type identities for $\pd_4(q)$.\medskip

Denote by $\varphi(q),\ \psi(q),$ and $f(a,b)$ Ramanujan's first, second, and general theta functions, respectively, defined by, cf. \cite[Entry 22(i), Entry 22(ii), Equation (18.1) and Entry 19]{Be91},
\begin{align*}
	\varphi(q)& =f(q,q)=\sum_{n\in\mathbb{Z}}q^{n^2}= \frac{(q^2;q^2)_\infty(-q;q^2)_\infty}{(q;q^2)_\infty(-q^2;q^2)_\infty}= \frac{f_2^2 (-q;q)_\infty}{f_1 (-q^2;q^2)^2_\infty}=\dfrac{f_2^5}{f_1^2 f_4^2},\\
	\psi(q) & = f(q,q^3)= \sum_{n\geq0}q^{n(n+1)/2}=\frac{(q^2;q^2)_\infty}{(q;q^2)_\infty} =\dfrac{f_2^2}{f_1},\\
	f(a,b) & =\sum_{n\in\mathbb{Z}}a^{n(n+1)/2}b^{n(n-1)/2}=(-a;ab)_\infty(-b;ab)_\infty(ab;ab)_\infty.
\end{align*}
Then, one can show the following identities by means of Theorem \ref{thm: pd4}.

\begin{theorem}
	Working modulo $2$, the following identities hold:
	\begin{align*}
		\quad\ \pd_4(q)\! \equiv\! \left[\dfrac{f_1^3}{f_3}\right]^3\! &\equiv \dfrac{f(q^6,q^{10})}{\psi(q)}+\dfrac{f(q^{18},q^{30})}{\psi(q^3)}-1 \notag \\
		&\equiv f_1^{13}(q^{6};q^{16})_\infty(q^{10};q^{16})_\infty+f_3^{13}(q^{18};q^{48})_\infty(q^{30};q^{48})_\infty-1 \notag \\
		& \equiv f_1^{13}(q^{6};q^{16})_\infty(q^{10};q^{16})_\infty+q^3f_3^{13}(q^{6};q^{48})_\infty(q^{42};q^{48})_\infty \notag \\
		& \equiv f_3^{13}(q^{18};q^{48})_\infty(q^{30};q^{48})_\infty+qf_1^{13}(q^{2};q^{16})_\infty(q^{14};q^{16})_\infty \, \operatorname{mod} 2.
	\end{align*}
\end{theorem}
\begin{proof}
	Theorem \ref{thm: pd4} claims
	\begin{align}\label{pd4 sum}
		\pd_4(q)= \frac{f_4f_6f_8f_{12}}{f_1f_2f_3f_{24}} \equiv\left[\dfrac{f_1^3}{f_3}\right]^3\!\equiv \dfrac{1}{2}\big[\varphi(q)+\varphi(q^2)+\varphi(q^3)+\varphi(q^6)-2\big] \,\!  \operatorname{mod} 2,
	\end{align}
	where we make repeated use of the identity $f_{2k} \equiv f_k^2 \, \operatorname{mod} 2$.
	The identities \cite[Example (iv), p. 51]{Be91} and \cite[Corollary (ii), p. 49]{Be91} imply, working modulo $4$ and $2$, respectively,
	\begin{align}\label{theta identity}
		\varphi(q)+\varphi(q^2)&\equiv \varphi(-q)+\varphi(q^2) = 2\dfrac{f^2(q^3,q^5)}{\psi(q)} \,  \operatorname{mod} 4,\\
		f^2(q^3,q^5)&\equiv f(q^6,q^{10})\equiv \psi(q)+qf(q^2,q^{14}) \,  \operatorname{mod} 2,\label{f identity}
	\end{align}
	while for $\psi(q)$ we have the identity
	\begin{align}\label{psi identity}
		\psi(q) = \dfrac{f_2^2}{f_1} \equiv f_1^3 \,  \operatorname{mod} 2.
	\end{align}
	Substituting Equations \eqref{theta identity} and \eqref{psi identity} into the right-hand side of Equation \eqref{pd4 sum} and applying the definition of $f(a,b)$, working modulo $2$ yields
	\begin{align}
		\pd_4(q)\equiv\left[\dfrac{f_1^3}{f_3}\right]^3 & \equiv \dfrac{1}{2}\big[\varphi(q)+\varphi(q^2)\big]+ \dfrac{1}{2}\big[\varphi(q^3)+\varphi(q^6)\big] -1 \notag \\
		&\equiv \dfrac{f^2(q^3,q^5)}{\psi(q)}+\dfrac{f^2(q^9,q^{15})}{\psi(q^3)}-1 \equiv \dfrac{f(q^6,q^{10})}{f_1^3}+\dfrac{f(q^{18},q^{30})}{f_3^3}-1 \notag \\
		&\equiv f_1^{13}(q^{6};q^{16})_\infty(q^{10};q^{16})_\infty+f_3^{13}(q^{18};q^{48})_\infty(q^{30};q^{48})_\infty-1 \, \operatorname{mod} 2. \notag
	\end{align}
	Now applying Equations \eqref{f identity} and  \eqref{psi identity} and the definition of $f(a,b)$, working modulo $2$ one obtains
	\begin{align}
		\pd_4(q)\equiv\left[\dfrac{f_1^3}{f_3}\right]^3& \equiv \dfrac{f^2(q^3,q^5)}{\psi(q)}+\dfrac{f^2(q^9,q^{15})}{\psi(q^3)}-1 \equiv \dfrac{f(q^6,q^{10})}{\psi(q)}+q^3\dfrac{f(q^6,q^{42})}{\psi(q^3)} \notag \\
		&\equiv f_1^{13}(q^{6};q^{16})_\infty(q^{10};q^{16})_\infty+q^3f_3^{13}(q^{6};q^{48})_\infty(q^{42};q^{48})_\infty \, \operatorname{mod} 2 \notag
	\end{align}
	and
	\begin{align}
		\pd_4(q)\equiv\left[\dfrac{f_1^3}{f_3}\right]^3& \equiv \dfrac{f^2(q^3,q^5)}{\psi(q)}+\dfrac{f^2(q^9,q^{15})}{\psi(q^3)}-1\equiv \dfrac{f(q^{18},q^{30})}{\psi(q^3)}+q\dfrac{f(q^2,q^{14})}{\psi(q)}\notag \\
		&\equiv f_3^{13}(q^{18};q^{48})_\infty(q^{30};q^{48})_\infty+qf_1^{13}(q^{2};q^{16})_\infty(q^{14};q^{16})_\infty \, \operatorname{mod} 2.\qedhere \notag
	\end{align}
\end{proof}

\section{The Case of $\pd_{3^\ell}(q) \, \operatorname{mod} 3$}\label{sec: mod 3}

\begin{definition}
	Let
	$$ h(q) = \frac{f_1(q)^2}{f_{2}(q)},$$
noting by Gauss's square power identity \cite[Equation (2.2.12)]{partitions},
	\begin{equation} \label{h}
		h(q) = \dfrac{(q;q)_\infty}{(-q;q)_\infty} = 1 + 2 \sum_{m \geq 1} (-1)^m q^{m^2}.
	\end{equation}
\end{definition}

Using the previous definition, we prove the following.

\begin{theorem} \label{pd 3 mod 3}
	Working modulo $3$,
	$$ \pd_3(q) \equiv \sum_{n \geq 0}e_{n1} q^n \, \operatorname{mod} 3 $$
	where
	\begin{align*}
		e_{n1} &= \big|\{\text{solutions to } n = k_0^2  + 3 k_1^2\, | \, \\
		&\quad\qquad\; k_m \in \mathbb{Z} \text{ or } \mathbb{N} \text{ when $k_m$ is even or odd, respectively}  \}\big|.
	\end{align*}
\end{theorem}

\begin{proof}
	Using $p = 3$ in Equation \eqref{pdk mod p}, we get
	$$ \pd_3(q) \equiv g(q)^2 \, \operatorname{mod} 3.$$
	As
	$$ g(q) = \frac{f_1(q) f_{2}(q) f_{3}(q)}{f_{6}(q)}
	\equiv h(q)^2 \, \operatorname{mod} 3, $$
	it follows, in a similar vein to Equation \eqref{pdk mod p}, that
	\begin{align} \label{pd 3 h}
	\pd_3(q) \equiv h(q)^4 = h(q) h(q)^3  \equiv h(q) h(q^3) \, \operatorname{mod} 3.
	\end{align}
	By Equation \eqref{h}, we may write
	$$ h(q) \equiv \sum_{m \in \mathbb{Z},\, 2\mid m} q^{m^2}
	+ \sum_{m\ge 1,\, 2\nmid m} q^{m^2} \, \operatorname{mod} 3.$$
	The result follows.
\end{proof}

\begin{remark}\label{remark pd3}
Note that $e_{n1}$ gives an interesting alternate characterization of \\
$\PD_3(n)  \, \operatorname{mod} 3$ to the one given in \cite[Theorem 2]{dS20b} where it was shown that, \\
for $n\ge 1$,
$$ \PD_3(n) \equiv \big|\{\text{solutions to } n = k(k + 1) + 3m(m+ 1) + 1 \, | \, k,m\ge 0\}\big| \, \operatorname{mod} 3.$$
Note also that we have
\[n = k(k + 1) + 3m(m+ 1) + 1 \iff 4n = (2k+1)^2 +3(2m+1)^2.\]
Thus, for $n\ge 1$, we can also write
	\begin{align*}
		\PD_3(n) & \equiv \big|\{\text{solutions to } 4n = k_0^2 +3k_1^2\, | \, k_m\ge 0, k_m \text{ odd}\}\big|\\
 &  \equiv \big|\{\text{solutions to } 4n = k_0^2 +3k_1^2\, | \, k_m \in \mathbb{Z}, k_m \text{ odd}\}\big| \, \operatorname{mod} 3.
	\end{align*}
\end{remark}

Theorem \ref{pd 3 mod 3} extends quite naturally as seen by the next result.

\begin{theorem} \label{our pd3}
For all $n\geq0$,
	$$ \PD_{3^\ell}(n) \equiv e_{n \ell} \, \operatorname{mod} 3 $$
	where
	\begin{align*}
		e_{n \ell} & =\Big|\{\text{solutions to } n = k_0^2 + \sum_{m=1}^{\ell -1} 3^m(k_m^2 + k_m'^2) + 3^\ell k_\ell^2\, | \, \\
		&\quad\qquad\; k_m, k_m' \in \mathbb{Z} \text{ or } \mathbb{N} \text{ when $k_m, k_m'$ is even or odd, respectively}  \}\Big|.
	\end{align*}
\end{theorem}

\begin{proof}
	Using Equations \eqref{pdk p to the ell} and \eqref{pd 3 h},
	$$ \pd_{3^\ell}(q) = \prod_{m = 0}^{\ell - 1} \pd_3(q^{3^m})
	\equiv  h(q) \, \bigg[\prod_{m = 1}^{\ell - 1} h(q^{3^m})^2 \bigg] \, h(q^{3^\ell}) \, \operatorname{mod} 3.   $$
	The result follows.
\end{proof}

As an easy application of Theorem \ref{our pd3}, we provide very short proofs of the following congruences for which proofs using dissections can be found in \cite[Theorem 3]{dS20b} and \cite[Theorem 3]{dS20a} for $\PD_3(9n+6)$ and $\PD_{3k}(3n+2)$, respectively.

\begin{theorem}\label{thm: cong 3}
For $n\geq0$,
\begin{align*}
\PD_{3}(9n+6)\equiv 0 \, \operatorname{mod} 3.
\end{align*}
For all $n\geq0$ and $\ell\geq1$,
\begin{align*}
\PD_{3^\ell}(3n+2)\equiv 0 \, \operatorname{mod} 3.
\end{align*}
\end{theorem}

\begin{proof}
For the first congruence, if $e_{n1}\neq0$ in Theorem \ref{our pd3}, then we can write, for suitable integers $k_0, k_1,$
\begin{align*}
n= k_0^2 + 3 k_1^2.
\end{align*}
A straightforward calculation shows that $n\not\equiv 6 \, \operatorname{mod} 9$ as $0,1,4,7$ are the only quadratic residues modulo 9. Hence, $e_{n1}=0$ for all $n\equiv6 \, \operatorname{mod} 9$.

The second congruence follows similarly. If $e_{n\ell}\neq0$ in Theorem \ref{our pd3}, then, since $0,1$ are the only quadratic residues modulo 3, we can write
\begin{align*}
n= k_0^2 + \sum_{m=1}^{\ell -1} 3^m(k_m^2 + k_m'^2) + 3^\ell k_\ell^2\equiv k_0^2 \not\equiv 2 \, \operatorname{mod} 3.
\end{align*}
Thus, $e_{n\ell}=0$ for all $n\equiv2 \, \operatorname{mod} 3$.
\end{proof}

As a further application of Theorem \ref{our pd3}, we prove the following newly observed congruences. We have not found the first congruence in the literature, but expect it to be known.

\begin{theorem}\label{thm: cong 4}
For $n\geq1$,
\begin{align*}
\PD_{3}(2n)\equiv 0 \, \operatorname{mod} 3.
\end{align*}
For all $n\geq0$ and $\ell\geq3$,
\begin{align*}
\PD_{3^\ell}(27n+9)\equiv 0 \, \operatorname{mod} 3.
\end{align*}
For all $n\geq0$ and $\ell\neq2$,
\begin{align*}
\PD_{3^\ell}(27n+18)\equiv 0 \, \operatorname{mod} 3.
\end{align*}
\end{theorem}

\begin{proof}
The first congruence is an immediate consequence of Theorem \ref{pd3} as
$n=k(k+1)+3m(m+1)+1$
is odd for nonnegative integers $k,m$.

The case $\ell=1$ of the third congruence follows similarly to that of the first congruence in the previous theorem. In particular, a straightforward calculation shows that $n=k_0^2+3k_1^2\not\equiv 18 \, \operatorname{mod} 27$ as $0, 1, 4, 7, 9, 10, 13, 16, 19, 22,$ and $25$ are the only quadratic residues modulo 27. Hence, $e_{n1}=0$ for all $n\equiv18 \, \operatorname{mod} 27$.

For $\ell\geq3$ in the third congruence, we argue as follows. We want to count the number of solutions to the Diophantine equation
\begin{align}\label{27n+18}
27n+18= k_0^2 + \sum_{m=1}^{\ell -1} 3^m(k_m^2 + k_m'^2) + 3^\ell k_\ell^2
\end{align}
modulo 3 with sign constraints as stated in Theorem \ref{pd 3 mod 3}. Note that we have $k_0^2\equiv 0 \, \operatorname{mod} 3$, thus $3\mid k_0$ and we can write $k_0=3a$ for some integer $a$. This gives
\begin{align}\label{27n+18A}
27n+18= 3(k_1^2+k_1'^2) + 9(a^2 + k_2^2 + k_2'^2) + \ldots + 3^\ell k_\ell^2.
\end{align}
In particular, $3(k_1^2+k_1'^2)\equiv 0 \, \operatorname{mod} 9$. Thus, $3\mid (k_1^2+k_1'^2)$, but as 0 and 1 are the only quadratic residues modulo 3, we conclude $3\mid k_1,k_1'$. Hence $3(k_1^2+k_1'^2)\equiv 0 \, \operatorname{mod} 27$, and Equation \eqref{27n+18A} implies
\begin{align*}
9(a^2 + k_2^2 + k_2'^2)\equiv 18 \, \operatorname{mod} 27,
\end{align*}
so that we conclude
\begin{align}\label{27n+18 mod 3}
a^2 + k_2^2 + k_2'^2\equiv 2 \, \operatorname{mod} 3.
\end{align}

It is immediate that, together with $(3a,k_1,k_1',k_2,k_2',k_3,k_3',\ldots,k_{\ell})$, also \linebreak $(3k_2,k_1,k_1',k_2',a,k_3,k_3',\ldots,k_{\ell})$ and $(3k'_2,k_1,k_1',a,k_2,k_3,k_3',\ldots,k_{\ell})$ solve Equation \eqref{27n+18}.
As a result, the mapping
\begin{align*}
\Psi((k_0,k_1,k_1',k_2,k_2',k_3,k_3',\ldots,k_{\ell}))=(3k_2,k_1,k_1',k_2',k_0/3,k_3,k_3',\ldots,k_{\ell}),
\end{align*}
defines a bijection on the solution set of \eqref{27n+18} with $\Psi^3=\text{id}$. The proof will be finished by showing that each orbit of the solution set under iterates of $\Psi$ has order~$3$. To see this, by way of contradiction, suppose
\begin{align*}
(3a,k_1,k_1',k_2,k_2',k_3,k_3',\ldots,k_{\ell})=(3k_2,k_1,k_1',k_2',a,k_3,k_3',\ldots,k_{\ell}).
\end{align*}
Comparing components yields $a=k_2=k_2'$, hence $a^2 + k_2^2 + k_2'^2\equiv 0 \, \operatorname{mod} 3$, contradicting \eqref{27n+18 mod 3} and completing the proof of the third congruence.

The second congruence follows analogously by systematically replacing 18 by 9 in the proof for $\ell\geq3$.
\end{proof}

It is noteworthy that we can recreate some of our structural results on $\pd_{2^\ell}(q)$ modulo $2$ from Section \ref{sec: mod 2} in the context of $\pd_{3^\ell}(q)$ modulo $3$. In particular,
the grouping argument of the last theorem once again can be iterated.

\begin{theorem}\label{thm: cong 4B}
For all $n\geq0$, $\ell \ge 2$, we have
\begin{align*}
\PD_{3^\ell}(3n)\equiv e^*_{n \ell}  \, \operatorname{mod}  3
\end{align*}
where
	\begin{align*}
		e^*_{n \ell} & = \big|\{\text{solutions to } n = k_0^2+k_0'^2+ 3^{\ell-1}(k_1^2 +k_1'^2) \, | \, \\
		&\quad\qquad\; k_m,k_m' \in \mathbb{Z} \text{ or } \mathbb{N} \text{ when $k_m,k_m'$ is even or odd, respectively}\}\big|.
	\end{align*}
\end{theorem}

\begin{proof}
We will be grouping the solutions $(k_0,k_1,k_1',k_2,k_2',k_3,k_3',\ldots,k_{\ell})$ of the equation
\begin{align}\label{3n 3A}
3n = k_0^2 + 3(k_1^2 +k_1'^2) + 9(k_2^2 +k_2'^2) + 27(k_3^2 +k_3'^2) +\ldots + 3^\ell k_{\ell}^2
\end{align}
into triplets. Note that we have $k_0^2 \equiv 0$ mod $3$, thus $3\mid k_0$ and we can write $k_0 = 3a$
for some integer $a$. This gives
\[
3n = 9a^2 + 3(k_1^2 +k_1'^2) + 9(k_2^2 +k_2'^2) + 27(k_3^2 +k_3'^2) +\ldots + 3^\ell k_{\ell}^2.
\]
Note that with $(3a,k_1,k_1',k_2,k_2',k_3,k_3',\ldots,k_{\ell})$ also $(3k_2,k_1,k_1',k_2',a,k_3,k_3',\ldots,k_{\ell})$
and $(3k_2',k_1,k_1',a,k_2,k_3,k_3',\ldots,k_{\ell})$ are solutions
of \eqref{3n 3A}, which provide us with a triplet of solutions unless $k_2 =k_2' =a$. This leaves us with considering
solutions of the form $(3a,k_1,k_1',a,a,k_3,k_3',\ldots,k_{\ell})$ and the equation
\[
3n = 9a^2 + 3(k_1^2 +k_1'^2) + 9(a^2 +a^2) + 27(k_3^2 +k_3'^2) +\ldots + 3^\ell k_{\ell}^2.
\]
Note that with $(3a,k_1,k_1',a,a,k_3,k_3',\ldots,k_{\ell})$ also $(3k_3,k_1,k_1',k_3,k_3,k_3',a,\ldots,k_{\ell})$
and $(3k_3',k_1,k_1',k_3',k_3',a,k_3,\ldots,k_{\ell})$ are solutions
of \eqref{3n 3A}, which provide us with a triplet of solutions unless $k_3 =k_3' =a$. This leaves us with considering
solutions of the form $(3a,k_1,k_1',a,a,a,a,k_4,k_4',\ldots,k_{\ell})$, and
this process can be iterated until we are left to consider solutions of the form $(3a,k_1,k_1',a,a,\ldots,a,a,k_{\ell})$
and the equation
\begin{align*}
3n & = 9a^2 + 3(k_1^2 +k_1'^2) + 9(a^2 +a^2) + 27(a^2 +a^2) +\ldots + 3^{\ell-1}(a^2+a^2) +3^\ell k_{\ell}^2\\
& = 3(k_1^2 +k_1'^2) + 3^\ell(a^2+k_{\ell}^2)
\end{align*}
In particular,
\[ n = k_1^2 +k_1'^2 + 3^{\ell-1}(a^2+k_{\ell}^2),\]
and the result follows.
\end{proof}

The next result can be proven the same way as Lemma \ref{lem: cong 1C}

\begin{lemma}\label{lem: cong 4C}
If there exist $j$ and $r$ such that the congruence
\begin{align*}
\PD_{3^j}(3^j n+r)\equiv0  \, \operatorname{mod}  3
\end{align*}
holds for all $n\ge 0$, then
\begin{align*}
\PD_{3^{\ell}}(3^j n+r)\equiv0  \, \operatorname{mod} 3
\end{align*}
for all $n\ge 0$ and $\ell \ge j$.
\end{lemma}

As an application of these results, we generalize Theorem \ref{thm: cong 4}.

\begin{corollary}\label{cor: cong 1D}
For all $n\geq0$, $k\ge 1$ and $\ell\geq 2k+1$,
\begin{align*}
\PD_{3^{\ell}}\big(3^{2k}(3n+1)\big)\equiv \PD_{3^{\ell}}\big(3^{2k}(3n+2)\big)\equiv 0 \, \operatorname{mod} 3.
\end{align*}
\end{corollary}

\begin{proof}
We will focus on the congruence $\PD_{3^{\ell}}\big(3^{2k}(3n+1)\big)\equiv 0 \, \operatorname{mod} 3$. With Lemma~\ref{lem: cong 4C},
it will be sufficient to prove $\PD_{3^{2k+1}}\big(3^{2k}(3n+1)\big)\equiv 0 \, \operatorname{mod} 3$, and according to Theorem \ref{thm: cong 4B}
we can focus on the equation
\[3^{2k-1}(3n+1) = k_0^2+k_0'^2+ 3^{2k}(k_1^2 +k_1'^2).\]
Note that $k_0^2+k_0'^2\equiv 0 \, \operatorname{mod} 3$. We conclude $3\mid k_0,k_0'$ and can write $k_0 = 3a_1, k_0' = 3a_1'$ for integers $a_1,a_1'$.
This leads to the new equation
\[3^{2k-3}(3n+1) = a_1^2+a_1'^2+ 3^{2k-2}(k_1^2 +k_1'^2).\]
We can iterate this argument to show that $k_0 = 3^ka_k, k_0' = 3^ka_k'$ with
\[3^{2k-1}(3n+1) = 3^{2k}(a_k^2 + a_k'^2+ k_1^2 +k_1'^2).\]
This is an apparent contradiction as the right-hand side of this equation is divisible by $3^{2k}$ while the left-hand side is not. This shows $e^*_{3^{2k}(3n+1), 2k+1} = 0$.
\end{proof}

\section{Open Problems} \label{concluding}



We conclude with some additional conjectured congruences.

\begin{conjecture}
For $n\geq0,$ we have
\begin{align*}
\PD_{2}(16n+12) \equiv0 \, \operatorname{mod} 4,\\
\PD_{2}(24n+20) \equiv0 \, \operatorname{mod} 4,\\
\PD_{2}(25n+5) \equiv0 \, \operatorname{mod} 4,\\
\PD_{2}(32n+24) \equiv0  \, \operatorname{mod} 4,\\
\PD_{2}(48n+26) \equiv0 \, \operatorname{mod} 4,
\end{align*}
and for $r\in\{5,11,15,17\}$,
\begin{align*}
\PD_9(54n+3r)\equiv0 \, \operatorname{mod} 3.
\end{align*}
\end{conjecture}



\end{document}